\documentclass[12pt, reqno]{amsart}

\usepackage{latexsym, amsmath, amssymb, amsthm, mathscinet, mathtools}
\usepackage{cases, verbatim}
\usepackage[active]{srcltx}
\usepackage{hyperref}
\usepackage{xcolor}
\usepackage{amsfonts}
\usepackage{pifont}
\usepackage{caption}
\usepackage{graphicx, subfig}
\usepackage{bm}
\usepackage{cite}
\usepackage{color,eucal,enumerate,mathrsfs}
\usepackage[normalem]{ulem}
\usepackage{amsmath}
\usepackage[pagewise]{lineno}
\usepackage[
  left=1in,right=1in,
  top=1in,bottom=1.2in,
  headsep=30pt,
  footskip=35pt
]{geometry}

\usepackage{fancyhdr}
\def\R{\mathbb R}
\def\N{\mathbb N}

\newtheorem{theorem}{Theorem}[section]

\newtheorem{remark}[theorem]{Remark}
\newtheorem{corollary}[theorem]{Corollary}
\newtheorem{proposition}[theorem]{Proposition}
\newtheorem{definition}[theorem]{Definition}
\newtheorem{example}[theorem]{Example}

\renewcommand{\d}{\mathrm{d}}
\newcommand{\mm}{\mathfrak m}

\newcommand{\mcp}{{\rm MCP}(0,N)}

\title[CKN on MMS]{$L^2$-Caffarelli--Kohn--Nirenberg inequalities on metric measure spaces}
\author{Zhe-Feng Xu$^{1,2}$, Ye Zhang$^1$}
\date{\today}

\begin{document}
\maketitle
\footnotetext[1]{SISSA, 34136, Trieste, Italy, zxu@sissa.it, yezhang@sissa.it} \footnotetext[2]{School of Mathematical Sciences, University of Science and Technology of China,  230026, Hefei, China, xzf1998@mail.ustc.edu.cn}
\begin{abstract}
Motivated by the sharp constants in the $L^2$-Caffarelli--Kohn--Nirenberg (or $L^2$-CKN for short) inequalities on Euclidean spaces, we study, in a unified framework, a sequence of $L^2$-CKN inequalities on metric measure spaces. On a general metric measure space, this sequence implies a reverse volume comparison of G\"unther type. Moreover, on a subclass of spaces admitting the measure contraction property, we show that this sequence of $L^2$-CKN inequalities are valid if and only if the spaces are volume cones. We also provide a stability result for inequalities of this type on volume cones.
\bigskip

\noindent
\emph{Keywords:} Bernstein's theorem, Caffarelli--Kohn--Nirenberg inequalities, Laplace transform, Needle decomposition

\noindent
\emph{MSC2020:} {\bf 44A10}, 28A50, 35A23
\end{abstract}

\renewcommand{\theequation}{\thesection.\arabic{equation}}
 \maketitle

\bigskip

\section{Introduction}
\subsection{Background}
This paper is primarily motivated by the sharp $L^{2}$-Caffarelli--Kohn--Nirenberg (or $L^2$-CKN for short) inequalities of the form
\begin{equation*}
\int_{\mathbb{R}^{N}} \frac{|\nabla u|^{2}}{|x|^{2b}} \, \mathrm{d}x
    \int_{\mathbb{R}^{N}} \frac{|u|^{2}}{|x|^{2a}} \, \mathrm{d}x \;
    \geq C^{2}(N,a,b)\left( \int_{\mathbb{R}^{N}} \frac{|u|^{2}}{|x|^{a+b+1}} \, \mathrm{d}x \right)^{2},
    \qquad u \in C_{0}^{\infty}(\mathbb{R}^{N}\setminus\{0\}),
\end{equation*}
which were established by Catrina and Wang \cite{CW2001} in the case $a=b+1$, by Costa \cite{C2008} for a particular range of parameters, and by Catrina and Costa \cite{CC2009} for the full range of parameters. More recently, Cazacu, Flynn and Lam \cite{CFL2021} presented a simple and direct approach that yields the best constant $C(N,a,b)$ for the entire admissible range of parameters and characterises all extremisers. The more general family of Caffarelli--Kohn--Nirenberg (or CKN for short) inequalities was first introduced and studied in \cite{CKN1982,CKN1984}, motivated in part by questions of well-posedness and regularity for certain Navier--Stokes equations. Moreover, various versions and extensions of the CKN inequalities have been investigated in broader settings, including Riemannian manifolds \cite{B2010,DX2004,M2015,N2022,X2007}, Finsler manifolds \cite{HKZ2020,KO2013,WW2020}, Lie groups \cite{F2020,H2015,ORS2019,RSY2017,Y2018}, and general metric measure spaces \cite{HX2024,K2025,KO2013,TAX2021}.

The main objective of the present paper is to investigate the $L^2$-CKN inequalities on more general metric measure spaces for a specific class of parameters. First, we recall the Euclidean version \cite{CC2009} for these parameters: let $k\geq 0$ and $N\geq 1$. Then, for any $u\in C^\infty_0(\R^N)$,
\begin{equation}\label{Euclidean}
\left( \int_{\R^N} |x|^{2k}\,|\nabla u|^{2}\, \d x \right)
\left( \int_{\R^N} |x|^{2k+2}\,|u|^{2}\, \d x  \right)
\geq \left(\frac{N}{2}+k\right)^2
\left( \int_{\R^N} |x|^{2k}\,|u|^{2}\, \d x  \right)^{2}
\end{equation}
holds. These inequalities are sharp, and equality is attained by $u(x)=ce^{-\lambda|x|^2}$ for any $c\in \R$ and $\lambda>0$.

\subsection{Main Results}
In this paper, $(X,\d)$ is a Polish space (i.e.\ a complete and separable metric space), and $\mm$ is a Radon measure on $X$ such that $0<\mm(U)<+\infty$ for any non-empty bounded open set $U\subset X$ (i.e.\ $\mathrm{supp}\,\mm=X$). The triple $(X,\d,\mm)$ is called a metric measure space.

For any fixed parameter $k\in\mathbb{N}$, we say that a metric measure space $(X,\d,\mm)$ admits the $L^2$-CKN inequalities (with parameter $k$) if there exists a constant $C(k)>0$ such that, for all
\[
u\in \mathrm{Lip}_b(X,\d) := \{ v:\ v \text{ is Lipschitz with bounded support}\},
\]
the inequality
\begin{equation}\label{first2}\tag*{$(\mathrm{CKN})_{x_0,k}'$}
\left( \int_X \d_{x_0}^{2k}\,|\text{lip}\, u|^{2}\, \mathrm{d}\mm \right)
\left( \int_X \d_{x_0}^{2k+2}\,|u|^{2}\, \mathrm{d}\mm \right)
\geq C(k)
\left( \int_X \d_{x_0}^{2k}\,|u|^{2}\, \mathrm{d}\mm \right)^{2}
\end{equation}
holds, where $\d_{x_0}(x):=\d(x_0,x)$ denotes the distance to a fixed point $x_0\in X$, and
\[
\text{lip}\, u(x):=\limsup_{y\to x}\frac{|u(y)-u(x)|}{\d(y,x)},\qquad x\in X
\]
stands for the local Lipschitz constant of $u$.

However, \ref{first2} for a single $k$ is not our main focus in this work. Instead, motivated by the sharp constants in \eqref{Euclidean}, we investigate the following counterpart, in which the constant depends on $k$ in a uniform manner. More precisely, we concentrate on the case where there exists a constant $C>0$, independent of $k$, such that for all $u\in \mathrm{Lip}_b(X,\d)$,
\begin{equation}\label{first}\tag*{$(\mathrm{CKN})_{C,x_0,k}$}
\left( \int_X \d_{x_0}^{2k}\,|\text{lip}\, u|^{2}\, \mathrm{d}\mm \right)
\left( \int_X \d_{x_0}^{2k+2}\,|u|^{2}\, \mathrm{d}\mm \right)
\geq (C + k)^2
\left( \int_X \d_{x_0}^{2k}\,|u|^{2}\, \mathrm{d}\mm \right)^{2}.
\end{equation}

\medskip
Our first main theorem can be stated as follows.
\begin{theorem}[See also Theorem \ref{thm Lap}]\label{Theorem1}
Let $(X,\d,\mm)$ be a metric measure space, and fix $x_0 \in X$. Assume that, for any $\lambda > 0$,
\[
\int_X e^{-2 \lambda \d_{x_0}^2} \, \d\mm < {+\infty},
\]
and that there exists a constant $C>0$, independent of $k$, such that for any $k \in \N$ and any $u\in \mathrm{Lip}_b(X,\d)$ the inequality \ref{first} holds.
Then, for any $\ell \in \N$, the function
\[
\rho \longmapsto \frac{\mm^\ell(B_\rho(x_0))}{\rho^{2(C + \ell)}}
\]
is non-decreasing on $(0,{+\infty})$. Here $B_{\rho}(x):=\{y\in X:\d(y,x)<\rho\}$, and the measure $\mm^\ell$ is defined by $\d \mm^\ell = \d_{x_0}^{2\ell} \d \mm$.
\end{theorem}

The motivation is as follows. The extremisers of \eqref{Euclidean} are precisely Gaussian functions. Roughly speaking, if we test \ref{first} with functions in the Gaussian class
\begin{equation}\label{Gauss}
    E_{G,x_0} := \{c e^{-\lambda \d_{x_0}^2} : \, c \in \R, \lambda > 0\},
\end{equation}
then, by the layer cake representation (or Fubini's theorem) together with a change of variables, we find that the Laplace transform of a certain function related to the volume of metric balls is non-negative. However, this does not directly imply that the function itself is non-negative. To obtain such a conclusion, we need to use Bernstein's theorem, whose assumptions are exactly \ref{first} for every $k\in \N$ with a $k$-independent constant $C>0$. For more details, we refer the reader to the proof of Theorem \ref{Theorem1}.

\begin{remark}\label{con}
It turns out that, in order to obtain the conclusion of Theorem \ref{Theorem1}, we genuinely need \ref{first} to hold for all $k\in \N$. In fact, in Subsection \ref{3.5} we shall present a counterexample for which \ref{first} holds for all $k\ge 1$ with a $k$-independent constant $C>0$ but fails for $k=0$; in this case, the function $\rho \longmapsto \mm(B_\rho(x_0))/\rho^{2C}$
is not non-decreasing.
\end{remark}

\begin{corollary}\label{prop1}
Let $\beta,\gamma>0$, $D\geq 1$, and $(X,\d,\mm)$ be a metric measure space, and fix $x_0 \in X$. Assume that there exists a constant $C>0$, independent of $k$, such that for any $k \in \N$ and any $u\in \mathrm{Lip}_b(X,\d)$ the inequality \ref{first} holds. Suppose also that
\begin{align}
\label{VD}\tag*{$(\mathrm{VD})_{D,\gamma}$}
\frac{\mm(B_{R}(x))}{R^\gamma} &\le D\,\frac{\mm(B_{r}(x))}{r^\gamma}, \qquad \forall \, 0<r<R,\, x \in X, \\
\label{AR}\tag*{$(\mathrm{AR})_{x_0, \beta, c_*}$}
\liminf_{r \to 0^+} &\frac{\mm(B_{r}(x_0))}{r^\beta} = c_* \in (0,+\infty).
\end{align}
Then $2C \le \beta \le \gamma$, and there exists a constant $C_*>0$ such that
\begin{align}\label{loweruni}
\frac{\mm(B_{r}(x))}{r^\beta}
\ge C_* \frac{1}{(1 + r + \d(x_0,x))^{\beta - 2C}}
\cdot\frac{1}{\left( 1 + \frac{\d(x_0,x)}{r}\right)^{\gamma - \beta}},
\qquad \forall \, r > 0,\, x \in X.
\end{align}
\end{corollary}

Next, we extend the rigidity result in \cite{HX2024} to $L^2$-CKN.
\begin{theorem}\label{Theorem2}
Let $N\in(1,{+\infty})$. Let $(X,\d,\mm)$ be an essentially non-branching metric measure space satisfying ${\rm MCP}(0,N)$. Then the following statements are equivalent:
\begin{itemize}
\item[{\bf (a)}] There exist $x_0\in X$ and $k\in\mathbb{N}$ such that \ref{first} holds with the sharp constant $C=\frac{N}{2}$;
\item[{\bf (b)}] There exists $x_0\in X$ such that, for every $k\in\mathbb{N}$, \ref{first} holds with the sharp constant $C=\frac{N}{2}$;
\item[{\bf (c)}] $(X,\d,\mm)$ is an $N$-volume cone.
\end{itemize}
\end{theorem}
In Theorem~\ref{Theorem2}, sharpness is understood in the sense that the sharp constant in \ref{first2} on $(X,\d,\mm)$ is $\left(\frac{N}{2}+k\right)^2$, as in the Euclidean space $\mathbb{R}^N$.

\medskip

Finally, motivated by the recent paper \cite{LLR2026}, we investigate the stability of \ref{first} on general metric measure spaces.

\begin{proposition}\label{stability}
Let $N\in(1,{+\infty})$ and $k \in \N$. Let $(X,\d,\mm)$ be an essentially non-branching metric measure space satisfying ${\rm MCP}(0,N)$, and assume that it is also an $N$-volume cone with vertex $x_0 \in X$. Then
\[
\delta_{k}(u) \ge \mathbb{D}_k^2(u,E_{G,x_0}), \qquad \forall \, u \in S_{x_0}^{r}.
\]
Here,
\[
\delta_{k}(u) := \left( \int_X \d_{x_0}^{2k}\,|\mathrm{lip}\, u|^{2}\, \mathrm{d}\mm \right)^{\frac{1}{2}}
\left( \int_X \d_{x_0}^{2k+2}\,|u|^{2}\, \mathrm{d}\mm \right)^{\frac{1}{2}}
- \left(\frac{N}{2}+k\right)
 \int_X \d_{x_0}^{2k}\,|u|^{2}\, \mathrm{d}\mm
\]
is the deficit, and
\[
\mathbb{D}_k^2(u,E_{G,x_0}) := \inf_{v \in E_{G,x_0}} \left\{ \int_X |u - v|^2 \d_{x_0}^{2k} \, \d \mm \right\},\quad
S_{x_0}^{r} := \{ u_0(\d_{x_0}) : u_0 \in C_0^\infty((0,+\infty)) \}.
\]
Moreover, $E_{G,x_0}$ is defined by \eqref{Gauss}.
\end{proposition}

For further stability results concerning other versions of the CKN inequalities, we refer the reader to \cite{CFLL2024,FP2024,WW2022,ZZ2026}.

\subsection{Some Generalisations}
In fact, our argument also applies to the $L^p$ analogues.
\begin{theorem}\label{Theorem1p}
Let $(X,\d,\mm)$ be a metric measure space and $p \in (1,+\infty)$. Fix $x_0 \in X$, and denote the conjugate exponent of $p$ by $p' := \frac{p}{p - 1}$.
Assume that, for any $\lambda > 0$,
\[
\int_X e^{-p \lambda \d_{x_0}^{p'}} \, \d\mm < {+\infty},
\]
and that there exists a constant $C>0$, independent of $k$, such that for any $k \in \N$ and any $u\in \mathrm{Lip}_b(X,\d)$ the following inequality holds:
\begin{equation}\label{firstp}
\left( \int_X \d_{x_0}^{p'k}\,|\mathrm{lip}\, u|^{p}\, \mathrm{d}\mm \right)
\left( \int_X \d_{x_0}^{p'(k + 1)}\,|u|^{p}\, \mathrm{d}\mm \right)^{p - 1}
\geq \left( \frac{ C + k}{p - 1}\right)^p
\left( \int_X \d_{x_0}^{p'k}\,|u|^{p}\, \mathrm{d}\mm \right)^{p}.
\end{equation}
Then, for any $\ell \in \N$, the function
\[
\rho \longmapsto \frac{\mm^{\ell,p}(B_\rho(x_0))}{\rho^{p'(C +\ell)}}
\]
is non-decreasing on $(0,{+\infty})$. Here the measure $\mm^{\ell,p}$ is defined by $\d \mm^{\ell,p} = \d_{x_0}^{p'\ell} \d \mm$.
\end{theorem}

\begin{remark}
We choose to write the constant in \eqref{firstp} as $\left( \frac{ C + k}{p - 1}\right)^p$ rather than $\left(  C+ \frac{k}{p - 1}\right)^p$ in order to make the proof more compatible with the proof of Theorem \ref{Theorem1}.
\end{remark}

Furthermore, for $p \in (1, +\infty)$ we can also obtain the analogues of Corollary \ref{prop1} (replacing $2$ by $p'$) and Theorem \ref{Theorem2} (with sharp constant $\frac{N}{p'}$), but we omit the statements here.

\subsection{Structure of the paper}
In Section \ref{section2} we recall the Bernstein's theorem and basic properties of the measure contraction property. In Section \ref{3.1} we prove Theorem \ref{Theorem1} by using the Bernstein's theorem. In Section \ref{3.2} we prove Corollary \ref{prop1}. In Section \ref{3.3} we prove Theorem \ref{Theorem2} by using the needle decomposition of measures. In Section \ref{3.4} we establish the stability result in Proposition \ref{stability}. Finally, in Section \ref{3.5} we provide the counterexample mentioned in Remark \ref{con}.

\vspace{0.3cm}

\section{Preliminaries}\label{section2}
\subsection{Bernstein's Theorem}\label{2.1}

One of our main tools in the proof of Theorem \ref{Theorem1} is the classical Bernstein's theorem, which states that any completely monotonic function on $(0,{+\infty})$ is the Laplace transform of a non-negative measure. For a detailed exposition of the Laplace transform and related topics, we refer the reader to the books \cite{SSV2012, W1941}.

\begin{definition}[Complete monotonicity]
A smooth function $f$ is completely monotonic on $(a,b)\subset \R$ ($-\infty \le a < b \le +\infty$) if it satisfies
\begin{align}\label{comm}
    (-1)^k f^{(k)}(x) \ge 0, \qquad \forall \, x \in (a,b), \, k \in \N.
\end{align}
\end{definition}

Next, we state the Bernstein's theorem, which can be found in \cite[Theorem 1.4]{SSV2012}.

\begin{theorem}[Bernstein's theorem]\label{tBern}
A smooth function $f$ is completely monotonic on $(0,{+\infty})$ if and only if it is the Laplace transform of a unique measure $\mu$ (a non-negative Borel measure), namely
\[
f(x) = \int_0^{+\infty} e^{-xt} \, \d\mu(t),
\]
and the integral converges for every $x \in (0,{+\infty})$.
\end{theorem}

\subsection{Measure Contraction Property (MCP)}\label{2.2}
Denote by
\begin{equation*}
	\text{Geo}(X):=\Big\{\gamma\in C([0,1],X):\d(\gamma_s,\gamma_t)=|s-t|\d(\gamma_0,\gamma_1), \,\,\forall s,t\in [0,1]\Big\}
\end{equation*}
the space of constant-speed geodesics. In this subsection, we assume that $(X,\d)$ is a geodesic space, meaning that for any $x,y \in X$ there exists $\gamma\in \text{Geo}(X)$ such that $\gamma_0=x$ and $\gamma_1=y$.

Denote by $\mathscr{P}(X)$ the space of all Borel probability measures on $X$ and by $\mathscr{P}_2(X)$ the space of probability measures with finite second moment. The $L^2$-Kantorovich--Wasserstein distance $W_2$ is defined as follows: for any $\mu_0,\mu_1 \in \mathscr{P}_2(X)$, set
\begin{equation}\label{equation2-1}
	W^2_2(\mu_0,\mu_1):= \inf\limits_{\pi} \int_{X\times X}  \d^2(x,y)\, \mathrm{d}\pi(x,y),
\end{equation}
where the infimum is taken over all $\pi \in\mathscr{P}(X\times X)$ with marginals $\mu_0$ and $\mu_1$.

For any geodesic $(\mu_t)_{t\in[0,1]}$ in $(\mathscr{P}_2(X),W_2)$, there exists $\nu \in \mathscr{P}(\text{Geo}(X))$ such that $(e_t)_\sharp \nu=\mu_t$ for all $t\in [0,1]$, where $e_t$ is the evaluation map
\begin{equation*}
	e_t:\text{Geo}(X) \rightarrow X,\quad e_t(\gamma):=\gamma_t.
\end{equation*}
We denote by $\mathrm{OptGeo}(\mu_0,\mu_1)$ the space of all $\nu \in \mathscr{P}(\text{Geo}(X))$ for which $(e_0,e_1)_\sharp \nu$ achieves the minimum in \eqref{equation2-1}; such a $\nu$ will be called a dynamical optimal plan. If $(X,\d)$ is geodesic, then $\mathrm{OptGeo}(\mu_0,\mu_1)$ is non-empty for any $\mu_0,\mu_1 \in \mathscr{P}_2(X)$.
\medskip

Recall the following notion of essentially non-branching \cite{MR3216835}.
\begin{definition}[Non-branching geodesic]
A set $G \subset \mathrm{Geo}(X)$ is a set of non-branching geodesics if, for any $\gamma^1,\gamma^2 \in G$, the following holds:
\[
\exists \, t \in (0,1)\ \, \text{such that}\ \,\forall \, s\in [0,t]\ \gamma_s^1 = \gamma_s^2
\Rightarrow
\forall \, s \in [0,1]\ \gamma_s^1 = \gamma_s^2.
\]
\end{definition}

\begin{definition}[Essentially non-branching space]
A metric measure space $(X,\d,\mm)$ is called essentially non-branching if, for any $\mu_0,\mu_1 \in \mathscr{P}_2(X)$ with $\mu_0, \mu_1\ll \mm$, every element of $\mathrm{OptGeo}(\mu_0,\mu_1)$ is concentrated on a set of non-branching geodesics.
\end{definition}

\medskip

The notion of the measure contraction property $\mathrm{MCP}(K,N)$ was proposed independently by Ohta and Sturm in \cite{MR2341840} and \cite{MR2237207} as a synthetic notion of lower Ricci curvature bounds. In general, these two definitions are slightly different, but on essentially non-branching spaces they coincide (see, for instance, Appendix~A in \cite{MR3691502} or Proposition~9.1 in \cite{MR4309491}). We adopt the one in \cite{MR2341840}. Here we only present the case $K=0$.

\begin{definition}[Measure contraction property]\label{define2-2}
	Let $N\in [1,{+\infty})$. We say that a metric measure space $(X,\d,\mm)$ satisfies $\mcp$ if for any $\mu_0 \in \mathscr{P}_2(X)$ of the form
\begin{equation*}
	\mu_0=\frac{1}{\mm(A)}\mm\llcorner_A,~~~A\subset X~\text{is Borel, and}~~\mm(A)\in(0,{+\infty})
\end{equation*}
and any $o\in X$, there exists $\nu \in \mathrm{OptGeo}(\mu_0,\delta_o)$ such that
\begin{equation*}
	\frac{1}{\mm(A)}\mm\geq  (e_t)_\sharp \left((1-t)^N\nu(\mathrm{d}\gamma) \right) \quad \forall \, t\in [0,1].
\end{equation*}
\end{definition}

From \cite{MR4309491}, we know that, in the setting of essentially non-branching spaces, Definition \ref{define2-2} is equivalent to the following: for all $\mu_0,\mu_1\in \mathscr{P}_2(X)$ with $\mu_0\ll \mm$, there exists a unique $\nu \in \mathrm{OptGeo}(\mu_0,\mu_1)$ such that, for all $t \in [0,1)$, $\mu_t=(e_t)_\sharp \nu\ll \mm$ and
\begin{equation}\label{mcp}
	\rho_t^{-\frac{1}{N}}(\gamma_t)\geq (1-t)\rho_0^{-\frac{1}{N}}(\gamma_0),\quad \text{for}\,\,\nu\text{-}a.e.\,\gamma \in \text{Geo}(X),
\end{equation}
where $\d \mu_t=\rho_t \d \mm$.

\medskip

In order to prove our main results, we need a corollary of the powerful needle decomposition theorem; here we only consider the simplest $1$-Lipschitz function $\d_{x_0}$. We refer to Cavalletti--Mondino \cite[Theorem 3.6]{MR4175820} for more general cases.

\begin{theorem}[Needle decomposition theorem]\label{theorem3}
	Let $(X,\d,\mm)$ be an essentially non-branching ${\rm MCP}(0,N)$ metric measure space for some $N\in(1,{+\infty})$. Then, for any $x_0\in X$ and $R>0$, there exist an $\mm$-measurable transport subset $\mathcal{T} \subset B_R(x_0)$ and a family $\{X_\alpha\}_{\alpha\in Q}$ of subsets of $B_R(x_0)$ such that
\begin{enumerate}
\item there exists a disintegration of $\mm$
\begin{equation*}
	\mm\llcorner_{\mathcal{T}}=\int_Q \mm_\alpha\, \d\mathfrak{q}(\alpha) ,\qquad \mathfrak{q}(Q)=1;
\end{equation*}
\item $\mm(B_R(x_0)\setminus \mathcal{T})=0$;
\item for $\mathfrak{q}$-a.e.\ $\alpha\in Q$, $X_\alpha$ is a closed geodesic with an extremal point $x_0$;
\item for $\mathfrak{q}$-a.e.\ $\alpha\in Q$, $\mm_\alpha$ is a Radon measure supported on $X_\alpha$ with $\mm_\alpha=h_\alpha \mathcal{H}^1\llcorner_{X_\alpha}$ and $h_\alpha \mathcal{H}^1\llcorner_{X_\alpha}\ll \mathcal{H}^1\llcorner_{X_\alpha}$;
\item for $\mathfrak{q}$-a.e.\ $\alpha\in Q$, the metric measure space $(X_\alpha,\d,\mm_\alpha)$ verifies ${\rm MCP}(0,N)$.
\end{enumerate}
Here $\mathcal{H}^1$ denotes the one-dimensional Hausdorff measure, $\{X_\alpha\}_{\alpha\in Q}$ are called transport rays, and two distinct transport rays can meet only at $x_0$.
\end{theorem}

It is worth recalling that, if $h_\alpha$ is an ${\rm MCP}(0,N)$ density on $I\subset \R$, then for all $x_0, x_1 \in I$ and $t\in [0,1]$,
\begin{equation}\label{1dmcp}
	h_\alpha(tx_1+(1-t)x_0)\geq (1-t)^{N-1}h_\alpha(x_0).
\end{equation}

At the end of this part, we recall the generalised Bishop--Gromov volume growth inequality (cf.\ \cite[Remark 5.3]{MR2237207}) and the definition of a volume cone.
\begin{theorem}[Generalised Bishop--Gromov inequality]\label{theorem4}
Let $N\in (1,+\infty)$, and $(X,\d,\mm)$ be a metric measure space satisfying ${\rm MCP}(0,N)$. Then, for any $x\in X$,
\begin{equation}\label{equation2-2}\tag{GBGI}
	\frac{\mm(B_{R}(x))}{R^N}
    \leq\frac{\mm(B_{r}(x))}{r^N},\quad \forall \, 0<r<R.
\end{equation}
\end{theorem}

\begin{definition}[$N$-volume cone]
Given $N\in [1,{+\infty})$, we say that a metric measure space is an $N$-volume cone if there exists $x_0\in X$ such that
\begin{equation*}
	\frac{\mm(B_{R}(x_0))}{\mm(B_{r}(x_0))}=\left(\frac{R}{r}\right)^N, \quad \forall \,0<r<R.
\end{equation*}
In this case, we also say that the metric measure space is an $N$-volume cone with vertex $x_0$.
\end{definition}

\vspace{0.3cm}

\section{Proof of main results}\label{proof}
\subsection{Proof of Theorem \ref{Theorem1}}\label{3.1}

It suffices to prove the case $\ell = 0$, because for any fixed $\ell \in \N \setminus \{0\}$ we still have
\[
\int_X e^{-2 \lambda \d_{x_0}^2} \, \d\mm^\ell < {+\infty},
\]
and, for any $k \in \N$, it holds that
\begin{align}\label{sequHUP3}
    \left(\int_X \d_{x_0}^{2k} \, |\mathrm{lip}\, u|^2 \, \d \mm^\ell \right)\left(\int_X \d_{x_0}^{2k + 2} |u|^2 \, \d \mm^\ell\right)
    \ge (C + \ell + k)^2 \left( \int_X \d_{x_0}^{2k} |u|^2 \, \d \mm^\ell \right)^2.
\end{align}
Consequently, we can work on $(X,\d,\mm^\ell)$ and apply the result for $\ell=0$ to this space. Noting that the constant $C$ changes to $C+\ell$, we see that the case $\ell \in \N \setminus \{0\}$ follows from the case $\ell=0$. Therefore, it remains to prove the case $\ell=0$.

For convenience, we restate Theorem \ref{Theorem1} with $\ell=0$ as follows.

\begin{theorem}\label{thm Lap}
Let $(X,\d,\mm)$ be a metric measure space, and fix $x_0 \in X$. Assume that, for any $\lambda > 0$,
\[
\int_X e^{-2 \lambda \d_{x_0}^2} \, \d\mm < {+\infty},
\]
and that there exists a constant $C>0$, independent of $k$, such that for any $k \in \N$ and any $u\in \mathrm{Lip}_b(X,\d)$,
\begin{align}\label{sequHUP2}
    \left(\int_X \d_{x_0}^{2k} \, |\mathrm{lip}\, u|^2 \, \d \mm \right)\left(\int_X \d_{x_0}^{2k + 2} |u|^2 \, \d \mm\right)
    \ge (C + k)^2 \left( \int_X \d_{x_0}^{2k} |u|^2 \, \d \mm \right)^2
\end{align}
holds.
Then the function $\rho \mapsto \frac{\mm(B_\rho(x_0))}{\rho^{2C}}$ is non-decreasing on $(0,{+\infty})$.
\end{theorem}

\begin{proof}
We define the auxiliary function
\[
    P(\lambda ) := \int_X e^{-2 \lambda \d_{x_0}^2} \, \d \mm.
\]
By assumption, $P$ is well defined on $(0,{+\infty})$. In what follows, we write $\overline{B}_{r}(x):=\{y\in X:\d(y,x)\le r\}$. Then, by Chebyshev's inequality, we have
\[
e^{-2 \lambda R^2}\,\mm(\overline{B}_R(x_0)) \le P(\lambda),\qquad \forall\, R > 0.
\]
Hence, for any $\delta > 0$, we have 
\begin{align}\label{growth}
\mm(\overline{B}_R(x_0)) \le P(\delta)\, e^{2 \delta R^2}, \qquad \forall\, R > 0.
\end{align}
Then, by Fubini's theorem and a change of variables, we obtain
\begin{align*}
  P(\lambda) &= \int_X \int_{\d_{x_0}}^{+\infty} 4\lambda \rho\, e^{-2\lambda \rho^2} \, \d\rho \, \d \mm \\
  &= 4 \lambda \int_0^{+\infty} \mm(\overline{B}_\rho(x_0))\,\rho\, e^{-2\lambda \rho^2} \, \d\rho \\
  &= \lambda \int_0^{+\infty} \mm\!\left(\overline{B}_{\sqrt{\frac{\tau}{2}}}(x_0)\right) e^{-\lambda \tau} \, \d \tau.
\end{align*}
Noting that the final integral is a Laplace transform, we conclude that $P$ is smooth on $(0,{+\infty})$. Define
\[
f_{x_0}(\tau) := \mm\left(\overline{B}_{\sqrt{\frac{\tau}{2}}}(x_0)\right),\qquad
Q(\lambda) := \frac{P(\lambda)}{\lambda},\qquad
R(\lambda) := -\lambda Q'(\lambda) - (C + 1) Q(\lambda).
\]
Using \eqref{growth} with $\delta = \frac{\lambda}{2}$, we have
\[
\lim_{\tau \to +\infty} \tau e^{-\lambda \tau} f_{x_0}(\tau) = 0, \qquad \forall \, \lambda > 0.
\]
Then, by integration by parts, we obtain
\begin{align*}
R(\lambda) &= -\lambda Q'(\lambda) - (C + 1) Q(\lambda) \\
&= -\int_0^{+\infty} f_{x_0}(\tau)\,(1-\tau \lambda)\, e^{-\lambda \tau} \, \d \tau
   - C \int_0^{+\infty} f_{x_0}(\tau) e^{-\lambda \tau} \, \d\tau \\
&= -\int_0^{+\infty} f_{x_0}(\tau)\, \d (\tau e^{-\lambda \tau})
   - C \int_0^{+\infty} f_{x_0}(\tau) e^{-\lambda \tau} \, \d\tau \\
&= \int_0^{+\infty} \tau e^{-\lambda \tau} \, \d f_{x_0}(\tau)
   - C \int_0^{+\infty} f_{x_0}(\tau) e^{-\lambda \tau} \, \d \tau.
\end{align*}

We claim that $R$ is completely monotonic, namely that $(-1)^k R^{(k)} \ge 0$ for any $k\in \N$. Once this is proved, Bernstein's theorem (Theorem \ref{tBern}) implies that there exists a non-negative Borel measure $\mu$ on $\R$ such that
\[
\tau \, \d f_{x_0}(\tau) - C f_{x_0}(\tau) \, \d\tau = \d \mu(\tau).
\]
It follows that for any $ 0<v \le w$,
\begin{align*}
   0 \le \int_v^w \tau^{-C - 1} \,  \d \mu(\tau)
   &= \int_v^w \tau^{-C } \,   \d f_{x_0}(\tau) -  C \int_v^w \tau^{-C - 1} f_{x_0}(\tau) \, \d\tau \\
   &= w^{-C} f_{x_0}(w) - v^{-C} f_{x_0}(v),
\end{align*}
where the last equality follows from integration by parts. As a result, we obtain
\[
\frac{f_{x_0}(v)}{v^C} \le \frac{f_{x_0}(w)}{w^C}, \qquad \forall \, 0<v \le w,
\]
which in turn implies
\[
\frac{f_{x_0}(v^-)}{v^C} \le \frac{f_{x_0}(w^-)}{w^C}, \qquad \forall \, 0<v \le w.
\]
From definition it is not hard to prove that
\[
f_{x_0}(v^-) :=  \lim_{u \to v^-} f_{x_0}(u) = \mm\!\left(B_{\sqrt{\frac{v}{2}}}(x_0)\right), \qquad \forall \, v \in (0,+\infty).
\]
Finally, for any $0<s \le t$, we have
\[
\frac{\mm(B_{s}(x_0))}{s^{2C}} \le \frac{\mm(B_{t}(x_0))}{t^{2C}},
\]
which is the desired result.

\medskip
Next, we are in a position to prove the claim. To this end, we introduce, for any $k\in \N$,
\[
S_k(\lambda) := (-1)^k P^{(k)}(\lambda) = \int_X (2 \d_{x_0}^2)^k e^{-2 \lambda \d_{x_0}^2} \, \d \mm \ge 0.
\]
For each $\lambda>0$, consider the sequence of functions $u_{\lambda,l} : X \rightarrow \R$, $l\in \N$, defined by
\begin{equation*}
	u_{\lambda,l}(x):=\max\big\{0,\min\{0,l-\d_{x_0}(x)\}+1\big\}e^{-\lambda \d_{x_0}^2(x)},
\end{equation*}
which belongs to $\mathrm{Lip}_b(X, \d)$. Set
\begin{equation*}
	u_\lambda(x):=\lim_{l\rightarrow {+\infty}} u_{\lambda,l}(x)=e^{-\lambda \d_{x_0}^2(x)}.
\end{equation*}
Using an approximation argument and the dominated convergence theorem, we readily see that $u_\lambda$ still satisfies \eqref{sequHUP2}. Inserting the function $u_\lambda = e^{-\lambda \d_{x_0}^2}$ into \eqref{sequHUP2}, and noting that
\[
|\text{lip}\, u_\lambda| = \big|2 \lambda \d_{x_0} e^{-\lambda \d_{x_0}^2} \,  \text{lip}\, \d_{x_0} \big|
\le 2 \lambda \d_{x_0} e^{-\lambda \d_{x_0}^2},
\]
we obtain
\begin{align}\label{have}
    \lambda S_{k + 1} (\lambda) \ge (C + k ) S_k(\lambda),  \qquad \forall\, k \in \N.
\end{align}
By the definition of $R$, we have
\[
R(\lambda)= -P'(\lambda)-C\frac{P(\lambda)}{\lambda},
\]
and a direct calculation yields
\begin{align}\label{desired}
(-1)^k R^{(k)}(\lambda) = S_{k + 1}(\lambda) - C \sum_{i = 0}^k \frac{k!}{i!}\frac{S_i(\lambda)}{\lambda^{k - i + 1}} .
\end{align}

To prove that $(-1)^k R^{(k)} \ge 0$, we argue by induction. For $k = 0$, \eqref{desired} reduces to \eqref{have}, so the claim holds. Now assume it holds for $k-1$ ($k\geq 1$), namely
\begin{align}\label{again}
   S_{k }(\lambda) \geq C \sum_{i = 0}^{k-1} \frac{(k-1)!}{i!}\frac{S_i(\lambda)}{\lambda^{k - i}},
\end{align}
Then, for $k$, combining \eqref{have} and \eqref{again}, we obtain
\begin{align*}
 S_{k + 1} (\lambda) \ge \frac{C}{\lambda} S_{k}(\lambda) +  \frac{k}{\lambda} S_k(\lambda) 
 \ge \frac{C}{\lambda} S_k(\lambda)  + \frac{k}{\lambda} \, C \sum_{i = 0}^{k - 1} \frac{(k-1)!}{i!}\frac{S_i(\lambda)}{\lambda^{k - i }}  
 = C \sum_{i = 0}^k \frac{k!}{i!}\frac{S_i(\lambda)}{\lambda^{k - i + 1}}.
\end{align*}
This completes the proof.
\end{proof}

\begin{remark}
In the proof above, $\d f_{x_0}$ denotes the measure induced by the right-continuous function $f_{x_0}$ (extended by $0$ to the whole of $\R$). This is the reason why we choose closed balls instead of open ones. We refer the reader to \cite[Section 1.5]{F1999} for details. Furthermore, for the integration by parts used in the proof, we refer the reader to \cite[Theorem 3.36]{F1999}.
\end{remark}

\begin{remark}
If \ref{first} holds only for a single $k \in \N$, then inserting functions of the form $\d_{x_0}^\beta e^{-\lambda \d_{x_0}^2}$ (after a standard approximation by truncation) yields some estimates, but these are not strong enough to obtain the conclusion of Theorem \ref{Theorem1}.
\end{remark}

To end this subsection, we also provide a sketch of the proof of Theorem \ref{Theorem1p}.

\begin{proof}[Sketch of the proof of Theorem \ref{Theorem1p}]
The proof is almost identical to that of Theorem \ref{Theorem1}. We can again reduce to the case $\ell = 0$. To treat the case $\ell = 0$, we define the new auxiliary function
\[
    P(\lambda ) := \int_X e^{-p \lambda \d_{x_0}^{p'}} \, \d \mm
    = \lambda \int_0^{+\infty} \mm\!\left(\overline{B}_{\left(\frac{\tau}{p}\right)^{1/p'}}(x_0)\right) e^{-\lambda \tau} \, \d \tau
\]
and the function
\[
f_{x_0}(\tau) := \mm\!\left(\overline{B}_{\left(\frac{\tau}{p}\right)^{1/p'}}(x_0)\right),
\]
instead of the original $P$ and $f_{x_0}$ in the proof of Theorem \ref{Theorem1}. The functions $Q$, $R$, and $S_k$ are modified accordingly using this choice of $P$ and $f_{x_0}$. As before, it remains to show that $R$ is completely monotonic. To this end, we insert the function (after a standard approximation by truncation) $u_\lambda := e^{-\lambda \d_{x_0}^{p'}}$ into \eqref{firstp}; the resulting estimate is exactly \eqref{have}. The argument then proceeds as above.
\end{proof}

\subsection{Proof of Corollary \ref{prop1}}\label{3.2}
\begin{proof}
Fix $x_0\in X$ and $r_0>0$. From \ref{VD} we know 
\[
\mm(B_{R}(x_0)) \le DR^\gamma\frac{\mm(B_{r_0}(x_0))}{r_0^\gamma}, \qquad \forall \, r_0<R.
\]
Then, by the layer cake representation and a change of variables, we have
\begin{align*}
\int_X e^{-2 \lambda \d_{x_0}^2} \, \d\mm
=&\int_{0}^{\infty} \mm\left(\left\{x \in X: e^{-2\lambda \d_{x_0}^2} > t\right\}\right)\d t \\
=& 4\lambda \int_{0}^{\infty} \mm\left(B_{\rho}(x_0)\right)\rho e^{-2\lambda \rho^2}\, \d\rho\\
\leq&\ 4\lambda \int_{0}^{r_0} \mm\left(B_{\rho}(x_0)\right)\rho e^{-2\lambda \rho^2}\, \d\rho
 + 4\lambda \int_{r_0}^{\infty} D\frac{\mm(B_{r_0}(x_0))}{r_0^\gamma}\rho^{\gamma+1} e^{-2\lambda \rho^2}\, \d\rho
 < +\infty.
\end{align*}
Thus, by Theorem \ref{Theorem1}, we obtain
\begin{equation}\label{RBG}
  \frac{\mm(B_{r}(x_0))}{r^{2C}}\leq \frac{\mm(B_{R}(x_0))}{R^{2C}}, \qquad \forall \, 0<r<R.
\end{equation}

We now prove the first assertion of Corollary \ref{prop1}. It follows from \ref{AR} that we can find a sequence $\{r_j\}_{j \in \N}$ such that $r_j \to 0^+$ and
\begin{align}\label{limit}
    \frac{1}{2}c_* r_j^\beta \le \mm(B_{r_j}(x_0)) \le 2c_* r_j^\beta, \qquad \forall \, j \in \N.
\end{align}
Note that \ref{VD} and \eqref{RBG} imply that
\[
D^{-1} r^\gamma \mm(B_{1}(x_0)) \le \mm(B_{r}(x_0)) \le r^{2C} \mm(B_{1}(x_0)), \qquad \forall \,r<1.
\]
Taking $r = r_j$ and combining \eqref{limit}, we obtain $2C \le \beta \le \gamma$.

We are now in a position to prove \eqref{loweruni}. We divide the proof into two cases. First, consider $x = x_0$. By \ref{AR}, there exists $\delta_0 > 0$ such that,
\[
\frac{\mm(B_{r}(x_0))}{r^\beta} \ge \frac{c_*}{2}, \qquad \forall \, r \in (0,\delta_0].
\]
If $r > \delta_0$, then by \eqref{RBG} we have
\[
\frac{\mm(B_{r}(x_0))}{r^\beta}
= \frac{\mm(B_{r}(x_0))}{r^{2C}}\, r^{2C - \beta}
\ge \frac{\mm(B_{\delta_0}(x_0))}{\delta_0^{2C}}\, r^{2C - \beta}.
\]
Therefore, there exists a constant $c(\delta_0,c_*)$ such that
\begin{align}\label{lowerx0}
    \frac{\mm(B_{r}(x_0))}{r^\beta} \ge c(\delta_0,c_*) \frac{1}{(1 + r)^{\beta - 2C}}, \qquad \forall \, r >0,
\end{align}
which is exactly \eqref{loweruni} in the case $x=x_0$.

We now turn to the case $x \ne x_0$. Fix $r>0$ and let $\rho > \max\{r,\d(x_0,x)\}$ be a number whose value will be determined later. Then \ref{VD} yields
\begin{align*}
  \frac{\mm(B_{r}(x))}{r^\beta}
  &= \frac{\mm(B_{r}(x))}{r^\gamma} r^{\gamma - \beta}
  \ge D^{-1} \frac{\mm(B_{\rho}(x))}{\rho^\gamma} r^{\gamma - \beta} \\
  &\ge D^{-1} \frac{\mm(B_{\rho - \d(x_0,x)}(x_0))}{\rho^\gamma} r^{\gamma - \beta} \\
  &= D^{-1} \frac{\mm(B_{\rho - \d(x_0,x)}(x_0))}{(\rho - \d(x_0,x))^\beta}
     \cdot\frac{(\rho - \d(x_0,x))^\beta}{\rho^\gamma} r^{\gamma - \beta} \\
  &\ge D^{-1} c(\delta_0,c_*)
     \frac{1}{(1 + \rho - \d(x_0,x))^{\beta - 2C}}\cdot
     \frac{(\rho - \d(x_0,x))^\beta}{\rho^\gamma} r^{\gamma - \beta},
\end{align*}
where in the second inequality we used the inclusion
$B_{\rho - \d(x_0,x)}(x_0) \subset B_{\rho}(x)$, and in the last inequality we used \eqref{lowerx0}. Finally, choosing $\rho = 2 \d(x_0,x) + r$ yields \eqref{loweruni}.
\end{proof}

\begin{remark}
Our choice of $\rho$ is optimal (in terms of the orders of $\d(x_0,x)$ and $r$) in the sense that there exists a constant $C_{**}$, independent of $x$ and $r$, such that
\[
\sup_{\rho > \max\{r,\d(x_0,x)\}} \frac{1}{(1 + \rho - \d(x_0,x))^{\beta -  {2C}}}\cdot \frac{(\rho - \d(x_0,x))^\beta}{\rho^\gamma} r^{\gamma - \beta}
\]
can be bounded from above by
\[
C_{**} \frac{1}{(1 + r + \d(x_0,x))^{\beta -  {2C}}} \cdot\frac{1}{\left( 1 + \frac{\d(x_0,x)}{r}\right)^{\gamma - \beta}}.
\]
\end{remark}

\subsection{Proof of Theorem \ref{Theorem2}}\label{3.3}
\begin{proof}
The proof is similar to \cite[Theorem 1.1]{HX2024}, so we adopt the same notation as in its proof and omit some details. First, the implication ${\bf (b) \Rightarrow (a)}$ is trivial.

\medskip

We now prove ${\bf (c) \Rightarrow (b)}$. Assume that the vertex of the $N$-volume cone is $x_0$. By a needle decomposition argument similar to Step~1 in \cite{HX2024}, we have
\begin{equation}\label{equation3-1}
\mm=\int_Q \mm_\alpha\, \d\mathfrak{q}(\alpha),\qquad (X_\alpha,\d,\mm_\alpha)\cong ([0, +\infty),|\cdot|,c_{\alpha}x^{N-1}\mathrm{d}x),
\end{equation}
where, for $\mathfrak{q}$-a.e.\ $\alpha \in Q$, $X_\alpha$ has $x_0$ as an extremal point. Fix $\alpha \in Q$ and denote by $\tilde{\Delta}_\alpha$ the weighted Laplacian
\[
\tilde{\Delta}_\alpha:=\Delta -\langle \nabla V_\alpha,\nabla \cdot\rangle,
\]
where $V_\alpha(x)$ is given by $e^{-V_\alpha(x)}=c_\alpha x^{N-1}$. Fix $u\in \mathrm{Lip}_b(X, \d)\setminus\{0\}$ and $k\in \N$. For $\d_0(x):=\d_{x_0}(x)=x$ we have
\[
\tilde{\Delta}_\alpha(\d_0^{2k+2})=(2k+2)(N+2k)\d_0^{2k},
\]
and therefore, by the Cauchy--Schwarz inequality,
\begin{equation*}
\begin{aligned}
&(2k+2)^2(N+2k)^2\left(\int_{X_\alpha}  \d_0^{2k}|u|^2 \, \d \mm_\alpha\right)^2\\
    &=\left(\int_{X_\alpha} \tilde{\Delta}_\alpha(\d_0^{2k+2}) |u|^2 \, \d \mm_\alpha\right)^2\\
    &=\left(-\int_{X_\alpha} \left<\nabla (|u|^2),\nabla (\d_0^{2k+2})\right>  \, \d \mm_\alpha\right)^2\\
&\leq 4(2k+2)^2\left(\int_{X_\alpha} \d_0^{2k}|\nabla u|^2\, \d \mm_\alpha\right)\left(\int_{X_\alpha} \d_0^{2k+2} |u|^2\, \d \mm_\alpha\right),
\end{aligned}
\end{equation*}
which implies
\begin{equation*}
\left(\int_{X_\alpha} \d_0^{2k}|\nabla u|^2\, \d \mm_\alpha\right)\left(\int_{X_\alpha} \d_0^{2k+2} |u|^2\, \d \mm_\alpha\right)
\geq\left(\frac{N}{2}+k\right)^2\left(\int_{X_\alpha}  \d_0^{2k}|u|^2 \, \d \mm_\alpha\right)^2.
\end{equation*}
Similar to Step~3 in \cite{HX2024}, we can adjust the decomposition and obtain
\begin{equation}
\left(\int_X \d^{2k}_{x_0}|\text{lip}\,u\,|^2\, \d\mm\right)
\left(\int_X \d_{x_0}^{2k+2} |u|^2\, \d\mm\right)
\geq\left(\frac{N}{2}+k\right)^2\left(\int_X \d^{2k}_{x_0}|u|^2\, \d\mm\right)^2,
\end{equation}
which is \ref{first} with vertex $x_0$ and $C=\frac{N}{2}$.
  
\medskip
For the sharpness of $\left(\frac{N}{2}+k\right)^2$, for each $\lambda>0$ consider the sequence of functions $u_{\lambda,l} : X \rightarrow \R$, $l\in \N$, defined by
\begin{equation*}
	u_{\lambda,l}(x):=\max\big\{0,\min\{0,l-\d_{x_0}(x)\}+1\big\} e^{-\lambda \d_{x_0}^2(x)},
\end{equation*}
which belongs to $\mathrm{Lip}_b(X, \d)$. Set
\begin{equation*}
	u_\lambda(x):=\lim_{l\rightarrow {+\infty}} u_{\lambda,l}(x)= e^{-\lambda \d_{x_0}^2(x)},
\end{equation*}
and, using an approximation argument similar to Step~4 in \cite{HX2024}, we complete the proof.

\medskip

We are left with the implication ${\bf (a) \Rightarrow (c)}$.
By Theorem \ref{theorem4} we know that $\rho \mapsto \frac{\mm(B_{\rho}(x_0))}{\rho^N}$ is non-increasing on $(0, {+\infty})$. Without loss of generality, we assume
\begin{equation}\label{equation3-12}
    \lim_{\rho \rightarrow 0^+} \frac{\mm(B_{\rho}(x_0))}{\omega_N \rho^{N}}=A,\quad A>0~~\text{is a finite constant},
\end{equation}
where $\omega_N:=\pi^{\frac{N}{2}}/\Gamma(\frac{N}{2}+1)$. Then we have
\begin{equation}\label{equation3-13}
	\mm(B_{\rho}(x_0))\leq A \omega_N \rho^{N},\qquad \forall \rho >0.
\end{equation}

Now we recall that the measure $\mm^k$ is given by the formula  $\d \mm^{k} := \d_{x_0}^{2k} \, \d \mm$. We claim that, for the measure $\mm^k$, we also have the analogous limit
\begin{equation}\label{equation3-12p}
    \lim_{\rho \rightarrow 0^+} \frac{\mm^{k}(B_{\rho}(x_0))}{\omega_N \rho^{N+2k}}=\frac{AN}{N+2k}
\end{equation}
and analogous upper bound
\begin{equation}\label{equation3-13p}
	\mm^{k}(B_{\rho}(x_0))\leq \frac{AN}{N+2k} \omega_N \rho^{N + 2k},\qquad \forall \rho >0.
\end{equation}
First, we show \eqref{equation3-13p}. Without loss of generality, we may assume $k \ge 1$. We shall use the following formula, which follows from Fubini's theorem:
\begin{align}\nonumber
    \mm^{k}(B_{\rho}(x_0)) &= \int_{B_{\rho}(x_0)} \d_{x_0}^{2k} \, \d \mm
    = 2k \int_{B_{\rho}(x_0)} \int_0^{\d_{x_0}(x)} t^{2k - 1} \, \d t \, \d \mm \\
\label{ecal}
&= 2k \int_0^\rho \bigl[\mm(B_{\rho}(x_0)) - \mm(B_{t}(x_0))\bigr] t^{2k - 1} \, \d t.
\end{align}
As a result, together with Theorem \ref{theorem4} and \eqref{equation3-13}, we obtain
\begin{equation}\label{small}
      \mm^{k}(B_{\rho}(x_0))
    \le 2k \mm(B_{\rho}(x_0)) \int_0^\rho \left( 1 - \frac{t^N}{\rho^N} \right) t^{2k - 1} \, \d t 
    \le \frac{AN}{N+2k} \omega_N \rho^{N+2k},
\end{equation}
which is exactly \eqref{equation3-13p}. On the other hand, it follows from \eqref{equation3-12} and \eqref{equation3-13} that for any $\varepsilon > 0$ there exists $\rho_\varepsilon>0$ such that
\[
(A- \varepsilon) \omega_N \rho^N \le \mm(B_{\rho}(x_0)) \le A \omega_N \rho^N, \qquad \forall \, \rho \in (0,\rho_\varepsilon].
\]
Then, for $\rho \in (0,\rho_\varepsilon]$ it follows from \eqref{ecal} that
\begin{align*}
    \mm^{k}(B_{\rho}(x_0))
    \ge 2k \int_0^\rho \bigl[(A- \varepsilon) \omega_N \rho^N - A \omega_N t^N\bigr] t^{2k - 1} \, \d t 
    = \left(\frac{AN}{N+2k} - \varepsilon \right)\omega_N \rho^{N + 2k}.
\end{align*}
Combining this with \eqref{small}, we obtain \eqref{equation3-12p}. With \eqref{equation3-12p} and \eqref{equation3-13p} in hand, using an argument similar to Step~5 in \cite{HX2024}, we can prove that $(X,\d,\mm^{k})$ is an $(N + 2k)$-volume cone. More precisely, we have
\[
\mm^{k}(B_{\rho}(x_0))=\frac{AN}{ N+2k} \omega_N \rho^{N + 2k}, \qquad \forall \rho>0.
\]
To prove that $(X,\d,\mm)$ is an $N$-volume cone, we use Fubini's theorem again:
\begin{align*}
    \mm(B_\rho(x_0))
    &=\int_{B_\rho(x_0)} \d_{x_0}^{-2k} \, \d \mm^{k} \\
    &=2k \int_{B_\rho(x_0)} \int_{\d_{x_0}(x)}^{+\infty} t^{-2k - 1} \, \d t \, \d \mm^{k} \\
    &=2k \int_0^{+\infty} \mm^{k} (B_{\min\{t,\rho\}}(x_0)) t^{-2k - 1} \, \d t \\
    &=\frac{2kAN}{N+2k} \omega_N \left( \int_0^\rho t^{N -1} \, \d t + \int_\rho^{+\infty} \rho^{N + 2k} t^{-2k -1} \, \d t \right)
    = A \omega_N \rho^N.
\end{align*}
Here, we can write \(\d \mm=\d_{x_0}^{-2k}\,\d \mm^k\) since \(\mm\) is non-atomic (due to MCP spaces) and thus $\mm(\{x_0\}) = 0$. Thus, $(X,\d, \mm)$ is an $N$-volume cone, and this completes the proof of Theorem \ref{Theorem2}.
\end{proof}

\subsection{Proof of Proposition \ref{stability}}\label{3.4}

\begin{proof}
We still use the needle decomposition method as in the proof of Theorem \ref{Theorem2}, and we obtain
\begin{equation*}
\mm=\int_Q \mm_\alpha\, \d\mathfrak{q}(\alpha),\qquad (X_\alpha,\d,\mm_\alpha)\cong ([0, +\infty),|\cdot|,c_{\alpha}x^{N-1}\mathrm{d}x),
\end{equation*}
where, for $\mathfrak{q}$-a.e.\ $\alpha \in Q$, $X_\alpha$ has $x_0$ as an extremal point.

Fix $u:=u_0(\d_{x_0})\in S^r_{x_0}$. Now, using \cite[Theorem 1.8]{LLR2026} (see also its proof), we know that
\begin{align*}
    &\left( \int_0^{+\infty} |u_0'|^2(x) x^{N + 2k - 1} \, \d x \right)^{\frac{1}{2}}\left( \int_0^{+\infty} |u_0|^2(x) x^{N + 2k + 1} \, \d x \right)^{\frac{1}{2}}
     \\
    &\quad - \left(\frac{N }{2} + k\right) \int_0^{+\infty} |u_0|^2(x) x^{N + 2k - 1} \, \d x
    \ge  \int_0^{+\infty} \left|u_0(x) - \eta \, e^{-\frac{x^2}{2\xi^2}} \right|^2 x^{N + 2k - 1} \, \d x,
\end{align*}
where
\begin{align*}
    \xi &:=  \xi(u_0) = \left( \frac{\int_0^{+\infty} |u_0|^2(x) x^{N + 2k + 1} \, \d x}{\int_0^{+\infty} |u_0'|^2(x) x^{N + 2k - 1} \, \d x} \right)^{\frac{1}{4}}, \\
    \eta &:=  \eta(u_0) = \frac{\int_0^{+\infty} u_0(x) x^{N + 2k - 1} e^{-\frac{x^2}{2\xi^2}}  \, \d x}{\int_0^{+\infty} x^{N + 2k - 1} e^{-\frac{x^2}{\xi^2}}  \, \d x},
\end{align*}
which are two constants independent of $\alpha$. Multiplying both sides by $c_\alpha$ and rewriting the inequality on $X_\alpha$ yield that for $\mathfrak{q}$-a.e.\ $\alpha \in Q$,
\begin{align*}
    &\left( \int_{X_\alpha} |u_0'|^2(\d_0) \d_0^{2k } \, \d \mm_\alpha \right)^{\frac{1}{2}}\left( \int_{X_\alpha} |u_0|^2(\d_0) \d_0^{2k + 2} \, \d \mm_\alpha \right)^{\frac{1}{2}}\\
    &\quad - \left(\frac{N }{2} + k\right) \int_{X_\alpha} |u_0|^2(\d_0) \d_0^{2k } \, \d \mm_\alpha
    \ge  \int_{X_\alpha} \left|u_0(\d_0) - \eta \, e^{-\frac{\d_0^2}{2\xi^2}} \right|^2 \d_0^{2k } \, \d \mm_\alpha.
\end{align*}
 Finally, integrating with respect to $\alpha$, applying the Cauchy--Schwarz inequality, and using the fact that $|u_0'|(\d_{x_0}) \le |\mathrm{lip}\, u|$, we obtain 
\begin{align*}
 &\left( \int_{X} |\mathrm{lip}\, u|^2 \d_{x_0}^{2k } \, \d \mm \right)^{\frac{1}{2}}\left( \int_{X} |u|^2 \d_{x_0}^{2k + 2} \, \d \mm \right)^{\frac{1}{2}}
    - \left(\frac{N }{2} + k\right) \int_{X} |u|^2 \d_{x_0}^{2k } \, \d \mm\\
    \geq&\left( \int_{X} |u_0'|^2(\d_{x_0}) \d_{x_0}^{2k } \, \d \mm \right)^{\frac{1}{2}}\left( \int_{X} |u_0|^2(\d_{x_0}) \d_{x_0}^{2k + 2} \, \d \mm \right)^{\frac{1}{2}}
    - \left(\frac{N }{2} + k\right) \int_{X} |u_0|^2(\d_{x_0}) \d_{x_0}^{2k } \, \d \mm\\
    \ge&  \int_{X} \left|u_0(\d_{x_0}) - \eta \, e^{-\frac{\d_{x_0}^2}{2\xi^2}} \right|^2 \d_{x_0}^{2k } \, \d \mm \geq\inf_{v \in E_{G,x_0}} \left\{ \int_X |u - v|^2 \d_{x_0}^{2k} \, \d \mm \right\}.
\end{align*}
So we finish the proof of Proposition \ref{stability}.
\end{proof}

\subsection{Counterexample}\label{3.5}
\begin{example}
Let $n\ge 1$. Set
\[
X=\mathbb R^n,\qquad \d(x,y)=|x-y|,\qquad x_0=0 \qquad (\mathrm{i.e.} \,\,\d_0(x) = |x|)
\]
and define a Radon measure on $X$ by
\[
\mm := \mathcal L^n + M\delta_{0},\qquad M>0,
\]
where $\mathcal L^n$ denotes the $n$-dimensional Lebesgue measure and $\delta_0$ the Dirac mass at $0$.
Then, for the choice of the constant $C = \frac{n}{2}$, \ref{first} holds for any $k\ge 1$ and any $u\in \mathrm{Lip}_b(X,\d)$, but it fails for $k=0$.
Moreover, the function $\rho \mapsto \frac{\mm(B_\rho(x_0))}{\rho^{2C}}$ cannot be non-decreasing on $(0,+\infty)$.
\end{example}

\begin{proof}
It is not hard to see that $(X,\d)$ is Polish, and $0<\mm(U)<+\infty$ for every non-empty bounded open set $U\subset\mathbb R^n$.
First, for every $\lambda>0$,
\[
\int_X e^{-2\lambda \d^2_0}\,\d \mm
= \int_{\mathbb R^n} e^{-2\lambda |x|^2}\,\d x + M <+\infty.
\]

Next, we check that \ref{first} holds for any $k\ge 1$ with $C=\frac n2$ and $x_0=0$.
Recall that $(\mathrm{CKN})_{C,0,k}$ reads: for any $u\in {\rm Lip}_b(X,\d)$,
\begin{equation}\label{eq:CKN}
\Big(\int_X \d_{0}^{2k}\,|{\rm lip}\,u|^2\,\d\mm\Big)
\Big(\int_X \d_{0}^{2k+2}\,|u|^2\,\d\mm\Big)
\ge (C+k)^2\Big(\int_X \d_{0}^{2k}\,|u|^2\,\d\mm\Big)^2.
\end{equation}
Fix $k\ge 1$. Since $\d^{2k}_0(0)=\d^{2k+2}_0(0)=0$, the contributions of the Dirac mass vanish, and \eqref{eq:CKN} is equivalent to the following: for any $u\in {\rm Lip}_b(\R^n)$, it holds that
\begin{equation*}
\Big(\int_{\R^n }|x|^{2k}\,|\nabla u|^2\,\d x\Big)
\Big(\int_{\R^n } |x|^{2k+2}\,|u|^2\,\d x\Big)
\ge \left(\frac{n}{2}+k\right)^2\Big(\int_{\R^n } |x|^{2k}\,|u|^2\,\d x\Big)^2,
\end{equation*}
which is the classical $L^2$-CKN inequality on $\R^n$ (see \eqref{Euclidean}); it follows, for instance, by approximation with functions in $C_0^{\infty}(\R^n)$.

However, for $k=0$, \eqref{eq:CKN} becomes
\begin{equation}\label{eq:CKN0}
\Big(\int_{\R^n } |\nabla u|^2\,\d\mm\Big)
\Big(\int_{\R^n } |x|^{2}\,|u|^2\,\d\mm\Big)
\geq C^2\Big(\int_{\R^n } |u|^2\,\d\mm\Big)^2,
\qquad \forall \, u\in {\rm Lip}_b(\R^n ).
\end{equation}
We claim that \eqref{eq:CKN0} fails for every $C>0$.
Choose a standard radial cut-off $u_\varepsilon\in{\rm Lip}_b(\mathbb R^n)$ such that, for any $\varepsilon\in(0,1)$,
\[
u_\varepsilon\equiv 1 \ \text{on }B_\varepsilon(0),\qquad
u_\varepsilon\equiv 0 \ \text{on }\mathbb R^n\setminus B_{2\varepsilon}(0),\qquad
|\nabla u_\varepsilon|\le \frac{2}{\varepsilon}.
\]
Then $u_\varepsilon(0)=1$, hence
\begin{equation}\label{right}
\int_{\R^n }|u_\varepsilon|^2\,\d \mm
\ge \int_{\R^n } |u_\varepsilon|^2\,\d(M\delta_0)
= M
\quad\Rightarrow\quad
C^2\Big(\int_{\R^n }|u_\varepsilon|^2\,\d \mm\Big)^2\ge C^2M^2.
\end{equation}
On the other hand, since $\mm=\mathcal L^n+M\delta_0$ and ${\rm supp}\,\nabla u_\varepsilon\subset  B_{2\varepsilon}(0)\setminus B_\varepsilon(0)$, we have
\begin{align}
\int_{\R^n } |\nabla u_{\varepsilon}|^2\,\d\mm
=\int_{B_{2\varepsilon}(0)\setminus B_\varepsilon(0)}|\nabla u_\varepsilon |^2\,\d x
\le \frac{4}{\varepsilon^2}\,\mathcal L^n(B_{2\varepsilon}(0))
= c_1\,\varepsilon^{n-2},
\label{eq:left1}
\end{align}
and
\begin{align}
\int_{\R^n} |x|^2 |u_\varepsilon|^2\,\d \mm
&=\int_{\mathbb R^n}|x|^2|u_\varepsilon|^2\,\d x
\le \int_{B_{2\varepsilon}(0)}|x|^2\,\d x
= c_2\,\varepsilon^{n+2},
\label{eq:left2}
\end{align}
for some constants $c_1,c_2>0$ depending only on $n$.
Multiplying \eqref{eq:left1} and \eqref{eq:left2} yields
\begin{equation}\label{eq:left}
\Big(\int_{\R^n } |\nabla u_{\varepsilon}|^2\,\d\mm\Big)
\Big(\int_{\R^n} |x|^2 |u_\varepsilon|^2\,\d \mm\Big)
\le c_1c_2\,\varepsilon^{(n-2)+(n+2)}
= c_1c_2\,\varepsilon^{2n}\xrightarrow[\varepsilon\to0]{}0,
\end{equation}
which contradicts \eqref{right}. Hence $(\mathrm{CKN})_{C, 0,0}$ fails for every $C>0$ on $(\mathbb R^n,|\cdot|,\mm)$.

Finally, we show that the monotonicity conclusion does not hold.
Indeed, for any $\rho>0$,
\[
\mm(B_\rho(0))
=\mathcal L^n(B_\rho(0))+M\delta_0(B_\rho(0))
=\omega_n\rho^n+M.
\]
Taking any $0< C \le \frac n2$, we obtain
\[
\frac{\mm(B_\rho(0))}{\rho^{2C}}
=\frac{\omega_n\rho^n+M}{\rho^{2C}}
=\omega_n \rho^{n - 2C}+\frac{M}{\rho^{2C}},\qquad \forall\,\rho>0,
\]
which is clearly not non-decreasing on $(0,+\infty)$.
\end{proof}

\section*{Acknowledgement}
\setcounter{equation}{0}

The first author is supported by the Young Scientist Programs of the Ministry of Science \& Technology of China (2021YFA1000900, 2021YFA1002200), the Shandong Provincial Natural Science Foundation (ZR2025QB05), and a China Scholarship Council grant (No.\allowbreak 202406340143).
The second author is supported by funding from the European Research Council (ERC) under the European Union's Horizon 2020 research and innovation programme (grant agreement GEOSUB, No.\allowbreak 945655) and the JSPS Grant-in-Aid for Early-Career Scientists (No.\allowbreak 24K16928).

\bibliography{bb} 

@article {LLR2026,
    AUTHOR = {Lam, Nguyen and Lu, Guozhen and Russanov, Andrey},
     TITLE = {Stability of {G}aussian {P}oincar\'e{} inequalities and
              {H}eisenberg uncertainty principle with monomial weights},
   JOURNAL = {Math. Z.},
  FJOURNAL = {Mathematische Zeitschrift},
    VOLUME = {312},
      YEAR = {2026},
    NUMBER = {2},
     PAGES = {Paper No. 42},
      ISSN = {0025-5874,1432-1823},
   MRCLASS = {26D10 (46E35 81Q05)},
  MRNUMBER = {5013418},
       DOI = {10.1007/s00209-025-03920-6},
       URL = {https://doi.org/10.1007/s00209-025-03920-6},
}

@book {F1999,
    AUTHOR = {Folland, Gerald B.},
     TITLE = {Real analysis},
    SERIES = {Pure and Applied Mathematics (New York)},
   EDITION = {Second},
      NOTE = {Modern techniques and their applications,
              A Wiley-Interscience Publication},
 PUBLISHER = {John Wiley \& Sons, Inc., New York},
      YEAR = {1999},
     PAGES = {xvi+386},
      ISBN = {0-471-31716-0},
   MRCLASS = {00A05 (26-01 28-01 46-01)},
  MRNUMBER = {1681462},
}

@book{SSV2012,
  title={Bernstein functions: theory and applications},
  author={Schilling, Ren{\'e} L and Song, Renming and Vondracek, Zoran},
  volume={37},
  year={2012},
  publisher={Walter de Gruyter}
}

@book {W1941,
    AUTHOR = {Widder, David Vernon},
     TITLE = {The {L}aplace {T}ransform},
    SERIES = {Princeton Mathematical Series},
    VOLUME = {vol. 6},
 PUBLISHER = {Princeton University Press, Princeton, NJ},
      YEAR = {1941},
     PAGES = {x+406},
   MRCLASS = {42.4X},
  MRNUMBER = {5923},
MRREVIEWER = {J.\ D.\ Tamarkin},
}

@article {CKN1984,
    AUTHOR = {Caffarelli, L. and Kohn, R. and Nirenberg, L.},
     TITLE = {First order interpolation inequalities with weights},
   JOURNAL = {Compos. Math.},
  FJOURNAL = {Compositio Mathematica},
    VOLUME = {53},
      YEAR = {1984},
    NUMBER = {3},
     PAGES = {259--275},
      ISSN = {0010-437X,1570-5846},
   MRCLASS = {46E35 (26D10)},
  MRNUMBER = {768824},
MRREVIEWER = {H.\ Triebel},
       URL = {http://www.numdam.org/item?id=CM_1984__53_3_259_0},
}

@article{CW2001,
  title={On the {C}affarelli-{K}ohn-{N}irenberg inequalities: sharp constants, existence (and nonexistence), and symmetry of extremal functions},
  author={Catrina, Florin and Wang, Zhi-Qiang},
  JOURNAL = {Commun. Pure Appl. Math.},
  FJOURNAL = {Communications on Pure and Applied Mathematics},
  volume={54},
  number={2},
  pages={229--258},
  year={2001},
  publisher={Wiley Online Library}
}

@article{C2008,
  title={Some new and short proofs for a class of {C}affarelli--{K}ohn--{N}irenberg type inequalities},
  author={Costa, David G},
  JOURNAL = {J. Math. Anal. Appl.},
  FJOURNAL = {Journal of Mathematical Analysis and Applications},
  volume={337},
  number={1},
  pages={311--317},
  year={2008},
  publisher={Elsevier}
}

@article{CC2009,
  title={Sharp weighted-norm inequalities for functions with compact support in $\mathbb{R}^N$$\setminus$$\{$0$\}$},
  author={Catrina, Florin and Costa, David G},
  JOURNAL = {J. Differential Equations},
  FJOURNAL = {Journal of Differential Equations},
  volume={246},
  number={1},
  pages={164--182},
  year={2009},
  publisher={Elsevier}
}

@article{CFL2021,
  title={Short proofs of refined sharp {C}affarelli-{K}ohn-{N}irenberg inequalities},
  author={Cazacu, Cristian and Flynn, Joshua and Lam, Nguyen},
  JOURNAL = {J. Differential Equations},
  FJOURNAL = {Journal of Differential Equations},
  volume={302},
  pages={533--549},
  year={2021},
  publisher={Elsevier}
}

@article{CKN1982,
  title={Partial regularity of suitable weak solutions of the {N}avier-{S}tokes equations},
  author={Caffarelli, Luis and Kohn, Robert and Nirenberg, Louis},
  JOURNAL = {Commun. Pure Appl. Math.},
  FJOURNAL = {Communications on Pure and Applied Mathematics},
  volume={35},
  number={6},
  pages={771--831},
  year={1982},
  publisher={Wiley Online Library}
}

@article{DX2004,
  title={Complete manifolds with non-negative {R}icci curvature and the {C}affarelli--{K}ohn--{N}irenberg inequalities},
  author={Do Carmo, Manfredo Perdig{\~a}o and Xia, Changyu},
  JOURNAL = {Compos. Math.},
  FJOURNAL = {Compositio Mathematica},
  volume={140},
  number={3},
  pages={818--826},
  year={2004},
  publisher={London Mathematical Society}
}

@article{X2007,
  title={The {C}affarelli-{K}ohn-{N}irenberg inequalities on complete manifolds},
  author={Xia, Changyu},
  JOURNAL = {Math. Res. Lett.},
  FJOURNAL = {Mathematical Research Letters},
  volume={14},
  number={5},
  pages={875--885},
  year={2007},
  publisher={International Press of Boston}
}

@article{B2010,
  title={A {C}affarelli--{K}ohn--{N}irenberg type inequality on {R}iemannian manifolds},
  author={Bozhkov, Yuri},
  JOURNAL = {Appl. Math. Lett.},
  FJOURNAL = {Applied Mathematics Letters},
  volume={23},
  number={10},
  pages={1166--1169},
  year={2010},
  publisher={Elsevier}
}

@article{M2015,
  title={The {C}affarelli--{K}ohn--{N}irenberg inequalities and manifolds with nonnegative weighted Ricci curvature},
  author={Mao, Jing},
  JOURNAL = {J. Math. Anal. Appl.},
  FJOURNAL = {Journal of Mathematical Analysis and Applications},
  volume={428},
  number={2},
  pages={866--881},
  year={2015},
  publisher={Elsevier}
}

@article{N2022,
  title={Sharp {C}affarelli--{K}ohn--{N}irenberg inequalities on {R}iemannian manifolds: the influence of curvature.},
  author={Nguyen, Van Hoang},
  JOURNAL = {Proc. Roy. Soc. Edinburgh Sect. A},
  FJOURNAL = {Proceedings of the Royal Society of Edinburgh. Section A. Mathematics},
  volume={152},
  number={1},
  year={2022}
}

@article {WW2020,
    AUTHOR = {Wei, Shihshu Walter and Wu, Bing Ye},
     TITLE = {Generalized {H}ardy type and {C}affarelli-{K}ohn-{N}irenberg
              type inequalities on {F}insler manifolds},
   JOURNAL = {Internat. J. Math.},
  FJOURNAL = {International Journal of Mathematics},
    VOLUME = {31},
      YEAR = {2020},
    NUMBER = {13},
     PAGES = {2050109, 27},
      ISSN = {0129-167X,1793-6519},
   MRCLASS = {53C60 (53B40)},
  MRNUMBER = {4192451},
MRREVIEWER = {Andrew\ Bucki},
       DOI = {10.1142/S0129167X20501098},
       URL = {https://doi.org/10.1142/S0129167X20501098},
}

@article{HKZ2020,
  title={Sharp uncertainty principles on general {F}insler manifolds},
  author={Huang, Libing and Krist{\'a}ly, Alexandru and Zhao, Wei},
  JOURNAL = {Trans. Amer. Math. Soc.},
  FJOURNAL = {Transactions of the American Mathematical Society},
  volume={373},
  number={11},
  pages={8127--8161},
  year={2020}
}

@article{KO2013,
  title={Caffarelli--{K}ohn--{N}irenberg inequality on metric measure spaces with applications},
  author={Krist{\'a}ly, Alexandru and Ohta, Shin-ichi},
  JOURNAL = {Math. Ann.},
  FJOURNAL = {Mathematische Annalen},
  volume={357},
  number={2},
  pages={711--726},
  year={2013},
  publisher={Springer}
}

@article{H2015,
  title={Weighted {C}affarelli-{K}ohn-{N}irenberg type inequality on the {H}eisenberg group},
  author={Han, Yazhou},
  JOURNAL = {Indian J. Pure Appl. Math.},
  FJOURNAL = {Indian Journal of Pure and Applied Mathematics},
  volume={46},
  number={2},
  pages={147--161},
  year={2015},
  publisher={Springer}
}

@article{ORS2019,
  title={${L}^p$-{C}affarelli--{K}ohn--{N}irenberg type inequalities on homogeneous groups},
  author={Ozawa, Tohru and Ruzhansky, Michael and Suragan, Durvudkhan},
  JOURNAL = {Q. J. Math.},
  FJOURNAL = {The Quarterly Journal of Mathematics},
  volume={70},
  number={1},
  pages={305--318},
  year={2019},
  publisher={Oxford University Press}
}

@article{RSY2017,
  title={Caffarelli--{K}ohn--{N}irenberg and {S}obolev type inequalities on stratified {L}ie groups},
  author={Ruzhansky, Michael and Suragan, Durvudkhan and Yessirkegenov, Nurgissa},
  JOURNAL = {Nonlinear Differential Equations Appl. NoDEA},
  FJOURNAL = {Nonlinear Differential Equations and Applications. NoDEA},
  volume={24},
  number={5},
  pages={56},
  year={2017},
  publisher={Springer}
}

@article{Y2018,
  title={Caffarelli--{K}ohn--{N}irenberg inequalities on {L}ie groups of polynomial growth},
  author={Yacoub, Chokri},
  JOURNAL = {Math. Nachr.},
  FJOURNAL = {Mathematische Nachrichten},
  volume={291},
  number={1},
  pages={204--214},
  year={2018},
  publisher={Wiley Online Library}
}

@article{F2020,
  title={Sharp {C}affarelli--{K}ohn--{N}irenberg-type inequalities on {C}arnot groups},
  author={Flynn, Joshua},
  JOURNAL = {Adv. Nonlinear Stud.},
  FJOURNAL = {Advanced Nonlinear Studies},
  volume={20},
  number={1},
  pages={95--111},
  year={2020},
  publisher={De Gruyter}
}

@article{TAX2021,
  title={The {C}affarelli--{K}ohn--{N}irenberg inequality on metric measure spaces},
  author={Tokura, Willian and Adriano, Levi and Xia, Changyu},
  JOURNAL = {Manuscripta Math.},
  FJOURNAL = {Manuscripta Mathematica},
  volume={165},
  number={1},
  pages={35--59},
  year={2021},
  publisher={Springer}
}

@article{HX2024,
  title={Sharp uncertainty principles on metric measure spaces},
  author={Han, Bang-Xian and Xu, Zhe-Feng},
  JOURNAL = {Calc. Var. Partial Differential Equations},
  FJOURNAL = {Calculus of Variations and Partial Differential Equations},
  volume={63},
  number={4},
  pages={104},
  year={2024},
  publisher={Springer}
}

@article{K2025,
  title={Volume growths versus {S}obolev inequalities},
  author={Krist{\'a}ly, Alexandru},
  JOURNAL = {Preprint arXiv:2501.16199},
  FJOURNAL = {arXiv e-prints},
  year={2025}
}

@article{WW2022,
  title={On the stability of the {C}affarelli--{K}ohn--{N}irenberg inequality},
  author={Wei, Juncheng and Wu, Yuanze},
  JOURNAL = {Math. Ann.},
  FJOURNAL = {Mathematische Annalen},
  volume={384},
  number={3},
  pages={1509--1546},
  year={2022},
  publisher={Springer}
}

@article{CFLL2024,
  title={Caffarelli--{K}ohn--{N}irenberg identities, inequalities and their stabilities},
  author={Cazacu, Cristian and Flynn, Joshua and Lam, Nguyen and Lu, Guozhen},
  JOURNAL = {J. Math. Pures Appl.},
  FJOURNAL = {Journal de Math{\'e}matiques Pures et Appliqu{\'e}es},
  volume={182},
  pages={253--284},
  year={2024},
  publisher={Elsevier}
}

@article{FP2024,
  title={Degenerate stability of the {C}affarelli--{K}ohn--{N}irenberg inequality along the {F}elli--{S}chneider curve},
  author={Frank, Rupert L and Peteranderl, Jonas W},
  JOURNAL = {Calc. Var. Partial Differential Equations},
  FJOURNAL = {Calculus of Variations and Partial Differential Equations},
  volume={63},
  number={2},
  pages={44},
  year={2024},
  publisher={Springer}
}

@article{ZZ2026,
  title={The stability on the {C}affarelli--{K}ohn--{N}irenberg and {H}ardy-type inequalities and beyond},
  author={Zhou, Yuxuan and Zou, Wenming},
  JOURNAL = {J. Differential Equations},
  FJOURNAL = {Journal of Differential Equations},
  volume={450},
  pages={113738},
  year={2026},
  publisher={Elsevier}
}

@article {MR3216835,
    AUTHOR = {Rajala, Tapio and Sturm, Karl-Theodor},
     TITLE = {Non-branching geodesics and optimal maps in strong {$CD(K,\infty)$}-spaces},
   JOURNAL = {Calc. Var. Partial Differential Equations},
  FJOURNAL = {Calculus of Variations and Partial Differential Equations},
    VOLUME = {50},
      YEAR = {2014},
    NUMBER = {3-4},
     PAGES = {831--846},
      ISSN = {0944-2669,1432-0835},
   MRCLASS = {53C23 (49Q20)},
  MRNUMBER = {3216835},
MRREVIEWER = {Nicolas\ Juillet},
       DOI = {10.1007/s00526-013-0657-x},
       URL = {https://doi.org/10.1007/s00526-013-0657-x},
}

@article {MR2341840,
    AUTHOR = {Ohta, Shin-ichi},
     TITLE = {On the measure contraction property of metric measure spaces},
   JOURNAL = {Comment. Math. Helv.},
  FJOURNAL = {Commentarii Mathematici Helvetici. A Journal of the Swiss Mathematical Society},
    VOLUME = {82},
      YEAR = {2007},
    NUMBER = {4},
     PAGES = {805--828},
      ISSN = {0010-2571,1420-8946},
   MRCLASS = {53C23 (28C15)},
  MRNUMBER = {2341840},
MRREVIEWER = {Jana\ Bj\"orn},
       DOI = {10.4171/CMH/110},
       URL = {https://doi.org/10.4171/CMH/110},
}

@article{MR2237207,
	AUTHOR = {K.-T. Sturm},
	TITLE = {On the geometry of metric measure spaces. {II}},
	JOURNAL = {Acta Math.},
	FJOURNAL = {Acta Mathematica},
	VOLUME = {196},
	YEAR = {2006},
	NUMBER = {1},
	PAGES = {133--177},
}

@article {MR3691502,
    AUTHOR = {Cavalletti, Fabio and Mondino, Andrea},
     TITLE = {Optimal maps in essentially non-branching spaces},
   JOURNAL = {Commun. Contemp. Math.},
  FJOURNAL = {Communications in Contemporary Mathematics},
    VOLUME = {19},
      YEAR = {2017},
    NUMBER = {6},
     PAGES = {1750007, 27},
      ISSN = {0219-1997,1793-6683},
   MRCLASS = {49Q20 (53C23)},
  MRNUMBER = {3691502},
MRREVIEWER = {Alexander\ O.\ Ivanov},
       DOI = {10.1142/S0219199717500079},
       URL = {https://doi.org/10.1142/S0219199717500079},
}

@article {MR4309491,
    AUTHOR = {Cavalletti, Fabio and Milman, Emanuel},
     TITLE = {The globalization theorem for the curvature-dimension
              condition},
   JOURNAL = {Invent. Math.},
  FJOURNAL = {Inventiones Mathematicae},
    VOLUME = {226},
      YEAR = {2021},
    NUMBER = {1},
     PAGES = {1--137},
      ISSN = {0020-9910,1432-1297},
   MRCLASS = {49Q22 (49Q20 53C23)},
  MRNUMBER = {4309491},
MRREVIEWER = {Luca\ Granieri},
       DOI = {10.1007/s00222-021-01040-6},
       URL = {https://doi.org/10.1007/s00222-021-01040-6},
}

@article {MR4175820,
    AUTHOR = {Cavalletti, Fabio and Mondino, Andrea},
     TITLE = {New formulas for the {L}aplacian of distance functions and
              applications},
   JOURNAL = {Anal. PDE},
  FJOURNAL = {Analysis \& PDE},
    VOLUME = {13},
      YEAR = {2020},
    NUMBER = {7},
     PAGES = {2091--2147},
      ISSN = {2157-5045,1948-206X},
   MRCLASS = {53C23 (49J52 49Q20 49Q22)},
  MRNUMBER = {4175820},
MRREVIEWER = {Emil\ Saucan},
       DOI = {10.2140/apde.2020.13.2091},
       URL = {https://doi.org/10.2140/apde.2020.13.2091},
}
\bibliographystyle{abbrv}


\end{document}